\documentclass[leqno,12pt]{article}
\usepackage{amsfonts}
\usepackage{amsmath, amssymb}
\usepackage{amsthm}

\theoremstyle{plain}
\newtheorem{theorem}{\indent\sc Theorem}[section]

\newtheorem{corollary}[theorem]{\indent\sc Corollary}
\newtheorem{proposition}[theorem]{\indent\sc Proposition}

\theoremstyle{definition}

\newtheorem{example}[theorem]{\indent\sc Example}

\newcommand\on{\operatorname}
\renewcommand\div{\on{div}}

\newcommand\Ric{\on{Ric}}
\newcommand\scal{\on{scal}}

\newcommand\func{\operatorname}
\newcommand\grad{\func{grad}}

\begin{document}

\title{Almost $\eta$-Ricci and almost $\eta$-Yamabe solitons\\
with torse-forming potential vector field}
\author{Adara M. BLAGA and Cihan \"{O}ZG\"{U}R}
\date{}
\maketitle

\begin{abstract}
We provide properties of almost $\eta$-Ricci and almost $\eta$-Yamabe
solitons on submanifolds isometrically immersed into a Riemannian manifold $%
\left(\widetilde{M},\widetilde{g}\right)$ whose potential vector field is
the tangential component of a torse-forming vector field on $\widetilde{M}$,
treating also the case of a minimal or pseudo quasi-umbilical hypersurface.
Moreover, we give necessary and sufficient conditions for an orientable
hypersurface of the unit sphere to be an almost $\eta$-Ricci or an almost $%
\eta$-Yamabe soliton in terms of the second fundamental tensor field.
\end{abstract}

\markboth{{\small\it {\hspace{4cm} Almost $\eta$-Ricci and $\eta$-Yamabe solitons}}}{\small\it{Almost $\eta$-Ricci and $\eta$-Yamabe solitons
\hspace{4cm}}}

%%%%%%%%%%%%%%% footnote %%%%%%%%%%%%%%%%
\footnote{%2010 MSC numbers
}
\footnote{%key words and phrases
}

\bigskip

\section{Introduction}

In 1982, R. S. Hamilton introduced the intrinsic geometric flows, \textit{%
Ricci flow} \cite{ham-82}
\begin{equation}
\frac{\partial }{\partial t}g(t)=-2\func{Ric}(g(t))
\end{equation}
and \textit{Yamabe flow} \cite{ham}
\begin{equation}
\frac{\partial }{\partial t}g(t)=-\func{scal}(t)\cdot g(t)
\end{equation}
which are evolution equations for Riemannian metrics. In a $2$-dimensional
manifold, Ricci flow and Yamabe flow are equivalent, but for higher
dimensions, we do not have such a relation.

Ricci solitons and Yamabe solitons correspond to self-similar solutions of
Ricci flow and Yamabe flow, respectively. Therefore, on an $n$-dimensional
smooth manifold $M$, a Riemannian metric $g$ and a non-vanishing vector
field $V$ is said to define \textit{a Ricci soliton} \cite{ham} if there
exists a real constant $\lambda$ such that
\begin{equation}  \label{1}
\frac{1}{2}\pounds _{V}g+\func{Ric}=\lambda g,
\end{equation}
respectively, \textit{a Yamabe soliton} \cite{ham} if there exists a real
constant $\lambda$ such that
\begin{equation}  \label{2}
\frac{1}{2}\pounds _{V}g=(\func{scal}-\lambda) g,
\end{equation}
where $\pounds _{V}$ denotes the Lie derivative operator in the direction of
the vector field $V$, $\func{Ric}$ and $\func{scal}$ denote the Ricci
curvature tensor field and respectively the scalar curvature of $g$. A Ricci
soliton (or a Yamabe soliton) $(V,\lambda)$ on a Riemannian manifold $(M,g)$
is said to be \textit{shrinking}, \textit{steady} or \textit{expanding}
according as $\lambda$ is positive, zero or negative, respectively.

Remark that Ricci solitons are natural generalizations of Einstein metrics,
any Einstein metric giving a trivial Ricci soliton.

\bigskip

Different generalizations of Ricci and Yamabe solitons have been lately
considered. If $\lambda$ is a smooth function on $M$, then (\ref{1}) defines
\textit{an almost Ricci soliton} \cite{pi} and (\ref{2}) defines \textit{an
almost Yamabe soliton} \cite{BB13}. Moreover, for a given $1$-form $\eta$ on
$M$, if there exist two real constants $\lambda$ and $\mu$ such that

(i)
\begin{equation}  \label{11}
\frac{1}{2}\pounds _{V}g+\func{Ric}=\lambda g+\mu \eta\otimes \eta
\end{equation}
we call $(V,\lambda,\mu)$ \textit{an $\eta$-Ricci soliton} \cite{ch}, and if

(ii)
\begin{equation}  \label{111}
\frac{1}{2}\pounds _{V}g=(\func{scal}-\lambda) g+\mu \eta\otimes \eta
\end{equation}
we call $(V,\lambda,\mu)$ \textit{an $\eta$-Yamabe soliton} \cite{cd}. If $%
\lambda$ and $\mu$ are smooth functions on $M$, then (\ref{11}) defines
\textit{an almost $\eta$-Ricci soliton} \cite{b} and (\ref{111}) defines
\textit{an almost $\eta$-Yamabe soliton} \cite{neto}.

\bigskip

If the potential vector field $V$ is of gradient type, $V=\func{grad}(\sigma
)$, for $\sigma $ a smooth function on $M$, then $(V,\lambda ,\mu )$ is
called a \textit{gradient soliton} (see \cite{ham}, \cite{pi}, \cite{ch} and
\cite{b}). If $\sigma $ is a constant, then the gradient soliton $(V,\lambda
,\mu )$ is trivial.

In \cite{ch}, J. T. Cho and M. Kimura classified a real hypersurface
admitting $\eta $-Ricci soliton in a non-flat complex space form. In \cite%
{Ch-10}, the same authors studied compact Lagrangian submanifolds in a K\"{a}%
hler manifold such that the induced metric on the Lagrangian submanifold is
a Ricci soliton with potential field $JH$, where $J$ is the complex
structure of the manifold and $H$ is the mean curvature vector field of the
submanifold. In \cite{Ch-12}, they considered complete Ricci solitons on
conformally flat hypersurfaces in Euclidean spaces and spheres. The
classification of all Ricci solitons on Euclidean hypersurfaces arisen from
the position vector field was given in \cite{Cd-14}, by B.-Y. Chen and S.
Deshmukh. A necessary and sufficient condition for an $n$-dimensional
submanifold $M$ to be a Ricci soliton with potential field $V^{T}$, where $%
V^{T}$ is the tangential component of $V$, was given in \cite{Chen-15}, by
B.-Y. Chen. Some necessary and sufficient conditions for a hypersurface of a
Euclidean space to be a gradient Ricci soliton were obtained in \cite{h}, by
{H. Al-Sodais, H. Alodan and S. Deshmukh}. In \cite{cd}, B.-Y. Chen and S.
Deshmukh considered Yamabe and quasi-Yamabe solitons on Euclidean
submanifolds whose soliton fields are the tangential components of their
position vector fields. In \cite{Eken-19}, \c{S}. E. Meri\c{c} and E. Kili%
\c{c} studied the conditions under which a submanifold of a Ricci soliton is
also a Ricci soliton or an almost Ricci soliton. The classification of
almost Yamabe solitons on hypersurfaces in Euclidean spaces arisen from the
position vector field was given in \cite{SM19}, by T. Seko and S. Maeta.

\bigskip

Motivated by the above studies, in the present paper, we establish some
properties of almost $\eta $-Ricci and almost $\eta $-Yamabe solitons on
Riemannian manifolds and on submanifolds isometrically immersed into a
Riemannian manifold $\left( \widetilde{M},\widetilde{g}\right) $ whose
potential vector field is the tangential component of a torse-forming (in
particular, concircular or recurrent) vector field on $\widetilde{M}$,
treating also the case of a minimal or pseudo quasi-umbilical hypersurface.
Moreover, we provide necessary and sufficient conditions for an orientable
hypersurface of the unit sphere to be an almost $\eta $-Ricci or an almost $%
\eta $-Yamabe soliton in terms of the second fundamental tensor field. A
partial study on this topic has been begun in \cite{bl}. Remark also that
almost Yamabe solitons on submanifolds were studied in \cite{SM19}.

\pagebreak 

\section{Preliminaries}

\medskip A non-flat Riemannian manifold $(M,g)$ $(n\geq 3)$ is called

a) \textit{mixed generalized quasi-Einstein manifold} \cite{bhat} if its
Ricci tensor field is not identically zero and verifies%
\begin{equation}
\func{Ric}=\alpha g+\beta A\otimes A+\gamma B\otimes B+\delta (A\otimes
B+B\otimes A),  \label{GQE}
\end{equation}%
where $\alpha ,\beta ,\gamma $ and $\delta $ are smooth functions and $A,$ $%
B $ are $1$-forms on $M$ such that the corresponding vector fields to the $1$%
-forms $A$ and $B$ are $g$-orthogonal. In particular, the manifold $M$ is
called:

\ \ \ i) \textit{generalized quasi-Einstein} \cite{Chaki-2001} if $\delta=0$;

\ \ \ ii) \textit{almost quasi-Einstein} \cite{chen-17} if $\beta=\gamma=0$;

\ \ \ iii) \textit{quasi-Einstein} \cite{ChakiMaity-2000} if $\beta=\delta=0$
or $\gamma=\delta=0$;

\ \ \ iv) \textit{Einstein} \cite{besse} if $\beta=\gamma=\delta=0$;

\medskip

b) \textit{pseudo quasi-Einstein} \cite{DeMa-2018} if its Ricci tensor field
is not identically zero and verifies
\begin{equation*}
\func{Ric}=\alpha g+\beta A\otimes A+\gamma E,
\end{equation*}%
where $\alpha ,\beta$ and $\gamma $ are smooth functions, $A$ is a $1$-form
and $E$ is a symmetric $(0,2)$-tensor field with vanishing trace on $M$. In
particular, if $\gamma=0$, then the manifold $M$ is \textit{quasi-Einstein}.

\medskip A vector field $V$ on a (pseudo)-Riemannian manifold $\left(
M,g\right) $ is called \textit{torse-forming }\cite{Yano-44} if
\begin{equation*}
\nabla _{X}V=aX+\psi (X)V,
\end{equation*}%
where $a$ is a smooth function, $\psi $ is a $1$-form and $\nabla $ is the
Levi-Civita connection of $g.$ Moreover, if $\psi (V)=0$, then $V$ is called
a \textit{torqued} vector field.

In particular, if $\psi =0$, $V$ is called a \textit{concircular} \textit{%
vector field }\cite{Fialkow-39} and if $a=0$, $V $ is called a \textit{%
recurrent} \textit{vector field }\cite{Shou}.

\section{Solitons with torse-forming potential vector field}

\subsection{Almost $\protect\eta$-Ricci solitons}

\bigskip Remark that any concircular vector field with $a(x)\neq 0$, for any
$x\in M$, is of gradient type, namely
\begin{equation*}
V=\frac{1}{2a}\func{grad}(|V|^{2})
\end{equation*}%
whose divergence is $\func{div}(V)=an$, with $n=\dim (M)$. Moreover,
\begin{equation*}
R(X,V)V=X(a)V-V(a)X
\end{equation*}%
for any $X\in \chi (M)$ and
\begin{equation*}
\func{Ric}(V,V)=(1-n)V(a).
\end{equation*}

If $(M,g)$ is an almost $\eta $-Ricci soliton with the potential vector
field $V$ and $\eta $ is the $g$-dual of $V$, then
\begin{equation*}
\func{Ric}=(\lambda -a)g+\mu \eta \otimes \eta ,
\end{equation*}%
\begin{equation*}
\func{scal}=(n-1)\left[ (\lambda -a)-\frac{V(a)}{|V|^{2}}\right]
\end{equation*}%
and we can state:

\begin{proposition}
\label{p} If a Riemannian manifold $(M,g)$ is an almost $\eta $-Ricci
soliton $(V,\lambda ,\mu )$ with concircular potential vector field $V$ and $%
a$ is a non zero constant, $\eta$ is the $g$-dual of $V$, then

i) $M$ is a quasi-Einstein manifold with associated functions $(\lambda-a)$
and $\mu$;

ii) $\func{grad}(\lambda)$, $\func{grad}(\mu)$ and $\func{grad}(\func{scal})$
are collinear with $V $.
\end{proposition}

\begin{proof}
Since $R(X,V)V=0$, for any $X\in \chi(M)$, we get $\Ric(V,V)=0$ and
$$(\lambda-a)+|V|^2\mu=0.$$

Differentiating the previous relation, using $d\scal =2\div(\Ric)$ and taking into account that
$$(\div(\mu(\eta\otimes \eta)))(X)=\mu g(\nabla_VV,X)+\mu\eta(X)\div(V)+\eta(X)d\mu(V),$$
we get:
$$(n-3)\grad (\lambda)=2[(n+1)a\mu+g(\grad(\mu),V)]V$$
and applying the gradient to the same relation, we have:
$$\grad(\lambda)=-2a\mu V-\grad(\mu)|V|^2,$$
which combined give:
$$\grad(\mu)=-\frac{2}{(n-3)|V|^2}[2(n-2)a\mu+g(\grad(\mu),V)]V.$$

Taking now the inner product with $V$, we get:
$$g(\grad(\mu),V)=-\frac{4(n-2)}{n-1}a\mu,$$
therefore
$$\grad(\mu)=-\frac{4(n-2)}{n-1}a\mu\frac{V}{|V|^2}.$$

Also, we get
$$\grad(\lambda)=\frac{2(n-3)}{n-1}a\mu V$$
and
$$\grad(\scal)=2(n-3)a\mu V.$$
\end{proof}

If $M$ has constant scalar curvature and $n>3$, then $\mu=0$, $\lambda=a$, $%
\func{Ric}=0$ and $\func{scal}=0$, therefore:

\begin{corollary}
Under the hypotheses of Proposition \ref{p}, if $M$ is of constant scalar
curvature and $n>3$, then $M$ is a Ricci-flat manifold.
\end{corollary}

For the almost Ricci solitons, we can state:

\begin{proposition}
\label{pp} If an $n$-dimensional Riemannian manifold $(M,g)$, with $n>3$, is
an almost Ricci soliton $(V,\lambda )$ with concircular potential vector
field $V$ and $a$ is a non zero constant, then $M$ is a Ricci-flat manifold.
\end{proposition}

\begin{proposition}
\label{p1} If a Riemannian manifold $(M,g)$ is an almost $\eta$-Ricci
soliton $(V,\lambda,\mu)$ with torqued potential vector field $V$ and $\eta$
is the $g$-dual of $V$, then $M$ is a mixed generalized quasi-Einstein
manifold with associated functions $(\lambda-a)$, $\mu$, $0$ and $-\frac{1}{2%
}$.
\end{proposition}

\begin{proof}
From the condition for $V$ to be torqued, we get
$$\frac{1}{2}(\pounds _{V}g)(X,Y)=\frac{1}{2}[g(\nabla_XV,Y)+g(\nabla_YV,X)]=ag(X,Y)+\frac{1}{2}[\psi(X)\eta(Y)+\eta(X)\psi(Y)]$$
and from the soliton equation, we obtain
$$\Ric(X,Y)=(\lambda-a)g(X,Y)+\mu \eta(X)\eta(Y)-\frac{1}{2}[\psi(X)\eta(Y)+\eta(X)\psi(Y)].$$

Let $U$ be the $g$-dual vector field of $\psi$. Then
$$\eta(U)=\psi(V)=0,$$
hence the conclusion.
\end{proof}

\begin{corollary}
Under the hypotheses of Proposition \ref{p1}, if $\psi=\mu \eta$, then $M$
is an Einstein manifold. In this case:

i) $V$ is a geodesic vector field if and only if $\mu=-\frac{a}{|V|^2}$;

ii) $V$ is concircular and the soliton is given by $(V,a,0)$.
\end{corollary}

\begin{proof}
From $\psi(V)=\mu \eta(V)$, we have
$$\nabla_VV=(a+\mu|V|^2)V,$$
hence the two statements.
\end{proof}

\medskip

If the potential vector field is torse-forming, we get the following two
results for almost $\eta $-Ricci solitons similar to those proved in \cite%
{chende} for Ricci solitons with concurrent potential vector field.

\begin{theorem}
Let $(M,g)$ be an $n$-dimensional Riemannian manifold and let $V$ be a
torse-forming vector field satisfying $\func{Ric}_{M}(V,V)=0$. Then $%
(V,\lambda ,\mu )$ is an almost $\eta $-Ricci soliton with $\eta$ the $g$%
-dual of $V$ if and only if the following conditions hold:

i) the soliton is given by $\left(\lambda, \psi(V)-\frac{\lambda-a}{|V|^2}%
\right)$;

ii) $M$ is an open part of a warped product manifold $(I\times
_{s}F,ds^{2}+s^{2}g_{F})$, where $I$ is an open real interval with arclength
$s$ and $F$ is an $(n-1)$-dimensional Einstein manifold with $\func{Ric}%
_{F}=(n-2)g_{F}$.
\end{theorem}

\begin{proof}
Following the same steps like in \cite{chende}, we get ii). In our case
$$\frac{1}{2}\pounds _{V}g=ag+\frac{1}{2}(\psi\otimes \eta+\eta\otimes \psi)$$
and from (\ref{11}), we have:
$$\Ric_M=(\lambda-a) g+\mu \eta \otimes \eta-\frac{1}{2}(\psi\otimes \eta+\eta\otimes \psi).$$
By using $\Ric_M(V,V)=0$, we obtain i).
\end{proof}

For almost Ricci solitons with concurrent potential vector field, we can
state:

\begin{proposition}
An $n$-dimensional Riemannian manifold $(M,g)$ is an almost Ricci soliton $%
(V,\lambda)$ with concurrent potential vector field $V$ if and only if the
following conditions hold:

i) the soliton is a shrinking Ricci soliton with $\lambda=1$;

ii) $M$ is an open part of a warped product manifold $(I\times
_{s}F,ds^{2}+s^{2}g_{F})$, where $I$ is an open real interval with arclength
$s$ and $F$ is an $(n-1)$-dimensional Einstein manifold with $\func{Ric}%
_{F}=(n-2)g_{F}$.
\end{proposition}

Therefore, there do not exist proper almost Ricci solitons (i.e. with
non-constant $\lambda $) with concurrent potential vector field.

\bigskip

It was proved in \cite{c} that the gradient of a non-constant smooth
function $\sigma $ on a Riemannian manifold $(M,g)$ is a torse-forming
vector field with
\begin{equation*}
\nabla _{X}(\func{grad}(\sigma ))=aX+\psi (X)\func{grad}(\sigma )
\end{equation*}%
if and only if
\begin{equation*}
\func{Hess}(\sigma )=ag+\delta d\sigma \otimes d\sigma ,
\end{equation*}%
where $\psi =\delta d\sigma $ and we can state:

\begin{proposition}
If a Riemannian manifold $(M,g)$ is an almost $\eta$-Ricci soliton $%
(V,\lambda,\mu)$ with torse-forming potential vector field $V=\func{grad}%
(\sigma)$ and $\eta$ is the $g$-dual of $V$, then $M$ is a quasi-Einstein
manifold with associated functions $(\lambda-a)$ and $(\mu-\delta)$.
\end{proposition}

\begin{proof}
From (\ref{11}), we get
$$\func{Ric}=(\lambda-a) g+(\mu-\delta) d\sigma \otimes d\sigma.$$ So $M$ is a quasi-Einstein manifold with
associated functions $(\lambda-a)$ and $(\mu-\delta)$.
\end{proof}

As a consequence, we obtain:

\begin{corollary}
With the above notations, if a Riemannian manifold $(M,g)$ is an almost $%
\eta $-Ricci soliton $(V,a,\delta)$ with torse-forming potential vector
field $V$ of gradient type and $\eta$ is the $g$-dual of $V$, then $M$ is a
Ricci-flat manifold.
\end{corollary}

\begin{example}
\label{ex1} \label{wexa} Let $M=\{(x,y,z)\in\mathbb{R}^3, z> 0\}$, where $%
(x,y,z)$ are the standard coordinates in $\mathbb{R}^3$. Set
\begin{equation*}
V:=-z\frac{\partial}{\partial z}, \ \ \eta:=-\frac{1}{z}dz,
\end{equation*}
\begin{equation*}
g:=\frac{1}{z^2}(dx\otimes dx+dy\otimes dy+dz\otimes dz).
\end{equation*}

Consider the linearly independent system of vector fields:
\begin{equation*}
E_1:=z\frac{\partial}{\partial x}, \ \ E_2:=z\frac{\partial}{\partial y}, \
\ E_3:=-z\frac{\partial}{\partial z}.
\end{equation*}
Then:
\begin{equation*}
\eta(E_1)=0, \ \ \eta(E_2)=0, \ \ \eta(E_3)=1,
\end{equation*}
\begin{equation*}
[E_1,E_2]=0, \ \ [E_2,E_3]=E_2, \ \ [E_3,E_1]=-E_1
\end{equation*}
and the Levi-Civita connection $\nabla$ is deduced from Koszul's formula
\begin{equation*}
2g(\nabla_XY,Z)=X(g(Y,Z))+Y(g(Z,X))-Z(g(X,Y))-
\end{equation*}%
\begin{equation*}
-g(X,[Y,Z])+g(Y,[Z,X])+g(Z,[X,Y]),
\end{equation*}
precisely
\begin{equation*}
\nabla_{E_1}E_1=-E_3, \ \ \nabla_{E_1}E_2=0, \ \ \nabla_{E_1}E_3=E_1, \ \
\nabla_{E_2}E_1=0,
\end{equation*}
\begin{equation*}
\nabla_{E_2}E_2=-E_3, \ \ \nabla_{E_2}E_3=E_2, \ \ \nabla_{E_3}E_1=0, \ \
\nabla_{E_3}E_2=0, \ \ \nabla_{E_3}E_3=0.
\end{equation*}
Then the Riemann and the Ricci curvature tensor fields are given by:
\begin{equation*}
R(E_1,E_2)E_2=-E_1, \ \ R(E_1,E_3)E_3=-E_1, \ \ R(E_2,E_1)E_1=-E_2,
\end{equation*}
\begin{equation*}
R(E_2,E_3)E_3=-E_2, \ \ R(E_3,E_1)E_1=-E_3, \ \ R(E_3,E_2)E_2=-E_3,
\end{equation*}
\begin{equation*}
\func{Ric}(E_1,E_1)=\func{Ric}(E_2,E_2)=\func{Ric}(E_3,E_3)=-2.
\end{equation*}
Writing the $\eta$-Ricci soliton equation in $(E_i,E_i)$ we obtain:
\begin{equation*}
g(\nabla_{E_i}E_3,E_i)+\func{Ric}(E_i,E_i)=\lambda
g(E_i,E_i)+\mu\eta(E_i)\eta(E_i),
\end{equation*}
for all $i\in\{1,2,3\}$ and we get that for $\lambda=\mu=-1$, the data $%
(V,\lambda,\mu)$ define an $\eta$-Ricci soliton on $(M, g)$. Moreover, it is
a gradient $\eta$-Ricci soliton, since the potential vector field $V$ is a
torse-forming vector field of gradient type, $V=\func{grad}(f)$, where $%
f(x,y,z):=-\ln z$.
\end{example}

\begin{example}
\label{ex2} Let $M=\mathbb{R}^3$, $(x,y,z)$ be the standard coordinates in $%
\mathbb{R}^3$ and $g$ be the Lorentzian metric:
\begin{equation*}
g:=e^{-2z}dx\otimes dx+e^{2x-2z}dy\otimes dy-dz\otimes dz.
\end{equation*}

Consider the vector field $V$ and the $1$-form $\eta$:
\begin{equation*}
V:=\frac{\partial}{\partial z}, \ \ \eta:=dz.
\end{equation*}

For the orthonormal vector fields:
\begin{equation*}
E_1=e^z\frac{\partial}{\partial x}, \ \ E_2=e^{z-x}\frac{\partial}{\partial y%
}, \ \ E_3=\frac{\partial}{\partial z},
\end{equation*}
we get:
\begin{equation*}
\nabla_{E_1}E_1=-E_3, \ \ \nabla_{E_1}E_2=0, \ \ \nabla_{E_1}E_3=-E_1, \ \
\nabla_{E_2}E_1=e^zE_2,
\end{equation*}%
\begin{equation*}
\nabla_{E_2}E_2=-e^zE_1-E_3, \ \ \nabla_{E_2}E_3=-E_2, \ \
\nabla_{E_3}E_1=0, \ \ \nabla_{E_3}E_2=0, \ \ \nabla_{E_3}E_3=0
\end{equation*}
and the Riemann tensor field and the Ricci tensor field are given by:
\begin{equation*}
R(E_1,E_2)E_2=(1-e^{2z})E_1, \ \ R(E_1,E_3)E_3=-E_1, \ \
R(E_2,E_1)E_1=(1-e^{2z})E_2,
\end{equation*}
\begin{equation*}
R(E_2,E_3)E_3=-E_2, \ \ R(E_3,E_1)E_1=E_3, \ \ R(E_3,E_2)E_2=E_3,
\end{equation*}
\begin{equation*}
\func{Ric}(E_1,E_1)=\func{Ric}(E_2,E_2)=2-e^{2z}, \ \ \func{Ric}(E_3,E_3)=-2.
\end{equation*}

Then the data $(V,\lambda,\mu)$, for $\lambda=1-e^{2z}$ and $\mu=-1-e^{2z}$,
define an almost $\eta$-Ricci soliton on $(M,g)$. Moreover, it is a gradient
almost $\eta$-Ricci soliton, since the potential vector field $V$ is of
gradient type, $V=\func{grad}(f)$, where $f(x,y,z):=-z$.
\end{example}

\subsection{Almost $\protect\eta$-Yamabe solitons}

Assume that $V$ is a torse-forming vector field on an $n$-dimensional
Riemannian manifold $(M,g)$, $\nabla_XV=aX+\psi(X)V$, for any $X\in \chi(M)$%
, where $\nabla$ is the Levi-Civita connection of $g$. If $\eta$ is the $g$%
-dual $1$-form of $V$ and $U$ is the $g$-dual vector field of $\psi$, then
\begin{equation*}
\eta(U)=\psi(V)
\end{equation*}
and
\begin{equation*}
\div(V)=an+\eta(U)=an+\psi(V).
\end{equation*}

Also
\begin{equation*}
\div(\psi\otimes \eta+\eta\otimes
\psi)=\div(V)\psi+\div(U)\eta+i_{\nabla_UV}g+i_{\nabla_VU}g.
\end{equation*}

Since
\begin{equation*}
\nabla_UV=aU+|U|^2V,
\end{equation*}
we obtain:
\begin{equation*}
\div(\psi\otimes \eta+\eta\otimes
\psi)=[(n+1)a+\psi(V)]\psi+[|U|^2+\div(U)]\eta+i_{\nabla_VU}g.
\end{equation*}

\bigskip

Let $(V,\lambda ,\mu)$ be an almost $\eta$-Yamabe soliton with $\eta$ the $g$%
-dual $1$-form of the torse-forming vector field $V$. Then
\begin{equation*}
\frac{1}{2}(\psi\otimes \eta+\eta\otimes \psi)=(\func{scal}-\lambda-a) g+\mu
\eta \otimes \eta
\end{equation*}
and taking the divergence, we get
\begin{equation}  \label{e3}
d(\func{scal}-\lambda-a)=\frac{1}{2}i_{\nabla_VU}g
\end{equation}
\begin{equation*}
+\frac{1}{2}[(n+1)a+\psi(V)]\psi+\left[\frac{|U|^2}{2}+\frac{\div(U)}{2}%
-(\mu+n)a-(\mu+1)\psi(V)-V(\mu)\right]\eta.
\end{equation*}

\begin{proposition}
Let $(M,g)$ be an $n$-dimensional Riemannian manifold and let $(V,\lambda
,\mu)$ be an almost $\eta$-Yamabe soliton with $\eta$ the $g$-dual $1$-form
of the torse-forming vector field $V$. Then $(V,\lambda ,\mu)$ is an almost
Yamabe soliton with $\lambda=\func{scal}-\frac{1}{2|V|^2}V(|V|^2)$ or $%
[\psi(V)]^2=|V|^2\cdot|U|^2$.
\end{proposition}

\begin{proof}
From the soliton equation (\ref{111}), taking $X=Y=V$, we get
\begin{equation}\label{e1}
\scal-\lambda-a-\psi(V)+\mu|V|^2=0
\end{equation}
and taking $X=Y=U$, we obtain
\begin{equation}\label{e2}
[\scal-\lambda-a-\psi(V)]|U|^2+\mu [\psi(V)]^2=0.
\end{equation}

Replacing (\ref{e1}) in (\ref{e2}), we get:
$$-\mu|V|^2|U|^2+\mu [\psi(V)]^2=0,$$
which implies $\mu=0$ (which yields an almost Yamabe soliton) or $[\psi(V)]^2=|V|^2\cdot|U|^2$.

If $\mu=0$, from (\ref{e1}), we get
$$\scal-\lambda-a-\psi(V)=0$$ and since
$$\frac{1}{2}V(|V|^2)=g(\nabla_VV,V)=[a+\psi(V)]|V|^2,$$
we obtain $\lambda=\scal-\frac{1}{2|V|^2}V(|V|^2)$.
\end{proof}

\begin{proposition}
Let $(M,g)$ be an $n$-dimensional Riemannian manifold and let $(V,\lambda
,\mu)$ be an almost $\eta$-Yamabe soliton with $\eta$ the $g$-dual $1$-form
of the concircular vector field $V$. Then $V$ is $\nabla$-parallel or the
soliton is given by
\begin{equation*}
(\lambda,\mu)=(\func{scal}-a+n|V|^2, n).
\end{equation*}
\end{proposition}

\begin{proof}
Taking $\psi=0$ and $U=0$ in (\ref{e3}), we get
$$d(\scal-\lambda-a)=[-(\mu+n)a-V(\mu)]\eta$$
and by differentiating (\ref{e1}), we obtain:
$$d(\scal-\lambda-a)=-\mu d(|V|^2)-|V|^2d\mu.$$
Replacing the second relation in the previous one and computing it in $V,$ we get
$$\mu V(|V|^2)=(\mu+n)a|V|^2.$$
Also
$$V(|V|^2)=2g(\nabla_VV,V)=2a|V|^2,$$
which combined with the previous relation implies either $a=0$ (i.e. $V$ is $\nabla$-parallel) or $\mu=n$, which from (\ref{e1}) gives $\lambda=\scal-a+n|V|^2$.
\end{proof}

\begin{proposition}
Let $(M,g)$ be an $n$-dimensional mixed generalized quasi-Einstein manifold
and let $(V,\lambda ,\mu )$ be an almost $\eta $-Yamabe soliton with $\eta $
the $g$-dual $1$-form of the torqued vector field $V$. If the Ricci tensor
field of $M$ is of the form $\func{Ric}=\alpha g+\beta
\eta\otimes\eta+\gamma ( \eta\otimes \psi+\psi\otimes\eta)$, then
\begin{equation}
\lambda =\left(\beta+\frac{\mu}{n}\right)|V|^2+\alpha n-a.  \label{e4}
\end{equation}
\end{proposition}

\begin{proof}
Since $V$ is a torqued vector field and $\eta $ the $g$-dual $1$%
-form of $V,$ we have
$$\left( \pounds %
_{V}g\right) (X,Y)=2ag(X,Y)+\eta (X)\psi(Y)+\psi(X)\eta (Y).$$ On the other
hand,%
\begin{equation*}
\func{Ric}(X,Y)=\alpha g(X,Y)+\beta \eta(X)\eta(Y)+\gamma [ \eta(X)\psi(Y)+\psi(X)\eta(Y)],
\end{equation*}%
gives us
\begin{equation*}
\scal=\alpha n+\beta |V|^2+2\gamma \psi(V)=\alpha n+\beta |V|^2.
\end{equation*}%
Then (\ref{111}) turns into
\begin{equation*}
ag(X,Y)+\frac{1}{2}\left[ \eta (X)\psi(Y)+\psi(X)\eta (Y)\right]
=[ \alpha n+\beta |V|^2-\lambda ] g(X,Y)+\mu \eta (X)\eta (Y).
\end{equation*}%
So by a contraction in the last equation, we obtain (\ref{e4}).
\end{proof}

\begin{example}
\label{exh} Let $M=\{(x,y,z)\in\mathbb{R}^3, z>0\}$, where $(x,y,z)$ are the
standard coordinates in $\mathbb{R}^3$. Set
\begin{equation*}
V:=-z\frac{\partial}{\partial z}, \ \ \eta:=-\frac{1}{z}dz,
\end{equation*}
\begin{equation*}
g:=\frac{1}{z^2}(dx\otimes dx+dy\otimes dy+dz\otimes dz)
\end{equation*}
and consider the linearly independent system of vector fields:
\begin{equation*}
E_1:=z\frac{\partial}{\partial x}, \ \ E_2:=z\frac{\partial}{\partial y}, \
\ E_3:=-z\frac{\partial}{\partial z}.
\end{equation*}
According to Example \ref{ex1}, the Riemann and the Ricci curvature tensor
fields are given by:
\begin{equation*}
R(E_1,E_2)E_2=-E_1, \ \ R(E_1,E_3)E_3=-E_1, \ \ R(E_2,E_1)E_1=-E_2,
\end{equation*}
\begin{equation*}
R(E_2,E_3)E_3=-E_2, \ \ R(E_3,E_1)E_1=-E_3, \ \ R(E_3,E_2)E_2=-E_3,
\end{equation*}
\begin{equation*}
\func{Ric}(E_1,E_1)=\func{Ric}(E_2,E_2)=\func{Ric}(E_3,E_3)=-2
\end{equation*}
and the scalar curvature is $\func{scal}=-6$. \newline
Writing the $\eta$-Yamabe soliton equation in $(E_i,E_i)$ we obtain:
\begin{equation*}
g(\nabla_{E_i}E_3,E_i)=(-6-\lambda)g(E_i,E_i)+\mu\eta(E_i)\eta(E_i),
\end{equation*}
for all $i\in\{1,2,3\}$ and we get that for $\lambda=-7$ and $\mu=-1$, the
data $(V,\lambda,\mu)$ define an $\eta$-Yamabe soliton on $(M, g)$.
Moreover, it is a gradient $\eta$-Yamabe soliton, since the potential vector
field $V$ is a torse-forming vector field of gradient type, $V=\func{grad}%
(f) $, where $f(x,y,z):=-\ln z$.
\end{example}

\begin{example}
Let $M=\mathbb{R}^3$, $(x,y,z)$ be the standard coordinates in $\mathbb{R}^3$%
, let $g$ be the Lorentzian metric:
\begin{equation*}
g:=e^{-2z}dx\otimes dx+e^{2x-2z}dy\otimes dy-dz\otimes dz.
\end{equation*}

Consider the vector field $V$ and the $1$-form $\eta$:
\begin{equation*}
V:=\frac{\partial}{\partial z}, \ \ \eta:=dz.
\end{equation*}

For the orthonormal vector fields:
\begin{equation*}
E_1=e^z\frac{\partial}{\partial x}, \ \ E_2=e^{z-x}\frac{\partial}{\partial y%
}, \ \ E_3=\frac{\partial}{\partial z},
\end{equation*}
according to Example \ref{ex2}, the Riemann tensor field, the Ricci tensor
field and the scalar curvature are given by:
\begin{equation*}
R(E_1,E_2)E_2=(1-e^{2z})E_1, \ \ R(E_1,E_3)E_3=-E_1, \ \
R(E_2,E_1)E_1=(1-e^{2z})E_2,
\end{equation*}
\begin{equation*}
R(E_2,E_3)E_3=-E_2, \ \ R(E_3,E_1)E_1=E_3, \ \ R(E_3,E_2)E_2=E_3,
\end{equation*}
\begin{equation*}
\func{Ric}(E_1,E_1)=\func{Ric}(E_2,E_2)=2-e^{2z}, \ \ \func{Ric}(E_3,E_3)=-2,
\end{equation*}
\begin{equation*}
\func{scal}=2(3-e^{2z}).
\end{equation*}

Then the data $(V,\lambda,\mu)$, for $\lambda=7-2e^{2z}$ and $\mu=-1$,
define an almost $\eta$-Yamabe soliton on $(M,g)$. Moreover, it is a
gradient almost $\eta$-Yamabe soliton, since the potential vector field $V$
is of gradient type, $V=\func{grad}(f)$, where $f(x,y,z):=-z$.
\end{example}

\section{Almost $\protect\eta$-Ricci and almost $\protect\eta$-Yamabe
solitons on submanifolds}

Let $M$ be a submanifold isometrically immersed into a Riemannian manifold $%
\left( \widetilde{M},\widetilde{g}\right) $. Denote by $g$ the Riemannian
metric induced on $M$, by $\nabla $ and $\widetilde{\nabla }$ the
Levi-Civita connections on $(M,g)$ and $\left( \widetilde{M},\widetilde{g}%
\right) $. The Gauss and Weingarten formulas corresponding to $M$ are given
by:
\begin{equation}
\widetilde{\nabla }_{X}Y=\nabla _{X}Y+h(X,Y),  \label{16}
\end{equation}%
\begin{equation}
\widetilde{\nabla }_{X}N=-B_{N}(X)+\nabla _{X}^{\bot }N,  \label{e5}
\end{equation}%
where $h$ is the second fundamental form and $B_{N}$ is the shape operator
in the direction of the normal vector field $N$ defined by $\widetilde{g}%
(B_{N}(X),Y)=\widetilde{g}(h(X,Y),N)$ for $X$, $Y\in \chi (M)$ \cite{Chen-73}%
.

\medskip

A Riemannian submanifold $M$ is called \textit{minimal} if its mean
curvature vanishes.

\medskip

A submanifold $M$ in a Riemannian manifold $\left( \widetilde{M},\widetilde{g%
}\right) $ is called

i) \textit{$\xi $-umbilical }\cite{Chen-16} (with respect to a normal vector
field $\xi $) if its shape operator satisfies $B_{\xi }=\varphi I,$ where $%
\varphi $ is a function on $M $ and $I$ is the identity map;

ii) \textit{totally umbilical }\cite{Chen-73} if it is umbilical with
respect to every unit normal vector field.

\medskip

An $n$-dimensional hypersurface $M$, $n\geq 4,$ in a Riemannian manifold $%
\left( \widetilde{M},\widetilde{g}\right) $ is called

i) $2$\textit{-quasi-umbilical }\cite{DeVers-79} if its second fundamental
tensor field $H$ satisfies
\begin{equation*}
H=\alpha g+\beta \varpi \otimes \varpi +\gamma \eta \otimes \eta ,
\end{equation*}%
where $\varpi $ and $\eta $ are $1$-forms and $\alpha $, $\beta $ and $%
\gamma $ are smooth functions on $M$ such that the corresponding vector
fields to the $1$-forms $\varpi $ and $\eta $ are $g$-orthogonal. In
particular, if $\gamma =0$, then $M$ is called \textit{quasi-umbilical} \cite%
{Chen-73};

ii) \textit{pseudo quasi-umbilical} \cite{DeMa-2018}, if its second
fundamental tensor field $H$ satisfies%
\begin{equation*}
H=\alpha g+\beta \varpi \otimes \varpi +E,
\end{equation*}%
where $\varpi $ is a $1$-form, $\alpha $ and $\beta $ are smooth functions
on $M$ and $E$ is a symmetric $(0,2)$-tensor field with vanishing trace. If $%
E$ vanishes, then $M$ is quasi-umbilical.

\medskip If $V$ is a concurrent vector field on $\widetilde{M}$, then for
any $X\in \chi (M)$, we have
\begin{equation*}
\nabla _{X}V^{T}=X+B_{V^{\bot }}(X),\ \ \nabla _{X}^{\bot }V^{\bot
}=-h(X,V^{T}),
\end{equation*}%
therefore
\begin{equation*}
\frac{1}{2}\func{grad}(|V^{\bot }|^{2})=-B_{V^{\bot }}(V^{T}),\ \ \frac{1}{2}%
\func{grad}(|V|^{2})=V^{T}
\end{equation*}%
and we can state:

\begin{proposition}
Every almost $\eta$-Ricci and every almost $\eta$-Yamabe soliton $%
(V^T,\lambda,\mu)$ on a submanifold $M$, which is isometrically immersed
into a Riemannian manifold $\left( \widetilde{M},\widetilde{g}\right) $, is
of gradient type with the potential function $\frac{1}{2}\func{grad}(|V|^2)$%
, where $V^T\in \chi(M)$ is the tangential component of the concurrent
vector field $V\in \chi(\widetilde{M})$ and $\eta$ is the $g$-dual of $V^T$.
\end{proposition}

In particular, if $V$ is of constant length, then the almost $\eta$-Ricci
soliton $(M,g)$ is a quasi-Einstein manifold with associated functions $%
\lambda$ and $\mu$.

\bigskip

Next, we shall characterize almost $\eta$-Ricci and almost $\eta$-Yamabe
solitons on $M$ whose potential vector field is the tangential component of
a torse-forming vector field on $\widetilde{M}$.

\begin{theorem}
\label{thm-sub}Let $M$ be a submanifold isometrically immersed into a
Riemannian manifold $\left( \widetilde{M},\widetilde{g}\right) $, let $V$ be
a torse-forming vector field on $\widetilde{M}$ and let $\eta $ be the $g$%
-dual of $V^T$. Then

i) $(M,g)$ is an almost $\eta $-Ricci soliton $(V^{T},\lambda ,\mu )$ if and
only if the Ricci tensor field of $M$ satisfies:
\begin{equation*}
\func{Ric}_{M}(X,Y)=(\lambda -a)g(X,Y)-\widetilde{g}(h(X,Y),V^{\perp })+\mu
\eta (X)\eta (Y)-\frac{1}{2}[\psi (X)\eta (Y)+\eta (X)\psi (Y)],
\end{equation*}%
for any $X,Y\in \chi (M)$;

ii) $(M,g)$ is an almost $\eta $-Yamabe soliton $(V^{T},\lambda ,\mu )$ if
and only if:
\begin{equation*}
(\func{scal}-\lambda-a)g(X,Y)-\widetilde{g}(h(X,Y),V^{\perp })+\mu \eta
(X)\eta (Y)-\frac{1}{2}[\psi (X)\eta (Y)+\eta (X)\psi (Y)]=0,
\end{equation*}%
for any $X,Y\in \chi (M)$.
\end{theorem}

\begin{proof}
For any $X\in \chi(M)$, we have:
$$aX+\psi(X)V^T+\psi(X)V^\bot=\widetilde{\nabla}_XV=\nabla_XV^T+h(X,V^T)-B_{V^{\bot}}(X)+\nabla^{\perp}_XV^{\bot}$$
and by the equality of the tangent components, we get:
$$\nabla_XV^T=aX+\psi(X)V^T+B_{V^{\bot}}(X).$$
Then
$$(\pounds _{V^T}g)(X,Y)=g(\nabla_XV^T,Y)+g(\nabla_YV^T,X)=2[ag(X,Y)+\widetilde{g}(h(X,Y),V^\bot)]$$$$+\psi(X)\eta(Y)+\eta(X)\psi(Y).$$

i) Suppose that there exist smooth functions $\lambda$ and $\mu$ on $M$ such that the condition in the hypotheses holds. Then we obtain
$$\frac{1}{2}(\pounds _{V^T}g)(X,Y)+\Ric_M(X,Y)=\lambda g(X,Y)+\mu \eta(X)\eta(Y).$$
Hence the submanifold $(M,g)$ is an almost $\eta$-Ricci soliton. The converse is trivial.

\pagebreak

ii) Suppose that there exist smooth functions $\lambda$ and $\mu$ on $M$ such that the condition in the hypotheses holds. Then we obtain
$$\frac{1}{2}(\pounds _{V^T}g)(X,Y)=(\scal-\lambda) g(X,Y)+\mu \eta(X)\eta(Y).$$
Hence the submanifold $(M,g)$ is an almost $\eta$-Yamabe soliton. The converse is trivial.
\end{proof}

If $M$ is a minimal submanifold, then we can state the following corollary:

\begin{corollary}
Let $M$ be an $n$-dimensional isometrically immersed minimal submanifold of
a Riemannian manifold $\left( \widetilde{M},\widetilde{g}\right) $, let $V$
be a concircular vector field on $\widetilde{M}$ and let $\eta $ be the $g$%
-dual of $V^{T}$.

i) If $(V^T,\lambda,\mu)$ is an almost $\eta$-Ricci soliton on $M$, then $%
\func{scal}_M=n(\lambda-a)+\mu |V^T|^2$.

ii) If $(V^T,\lambda,\mu)$ is an almost $\eta$-Yamabe soliton on $M$, then $%
\func{scal}_M=\lambda+a-\frac{\mu}{n}|V^T|^2$.
\end{corollary}

When $M$ is a $V^{\bot }$-umbilical submanifold, we have:

\begin{corollary}
Let $M$ be an $n$-dimensional $V^{\bot }$-umbilical submanifold
isometrically immersed into an $(n+d)$-dimensional Riemannian manifold $%
\left( \widetilde{M},\widetilde{g}\right) .$ If $V$ is a concircular vector
field on $\widetilde{M},$ then $M$ is an almost $\eta $-Ricci soliton with
potential vector field $V^{T}$, for $\eta $ the $g$-dual of $V^{T}$ if and
only if $M$ is a quasi-Einstein submanifold with associated functions $%
\left( \lambda -a-\varphi \right) $ and $\mu .$
\end{corollary}

For a hypersurface, since
\begin{equation*}
\widetilde{g}\left( h(X,Y),V^{\perp }\right) =g(B(X),Y)g(N,V^{\perp
})=H(X,Y)g(N,V^{\perp }),
\end{equation*}
where $N$ is the unit normal vector field of $M$ and $H$ is the second
fundamental tensor field, if we denote by $\rho =g(N,V^{\perp }),$ then we
can state:

\begin{corollary}
\label{corA}Let $M$ be an $n$-dimensional hypersurface isometrically
immersed into an $(n+1)$-dimensional Riemannian manifold $\left( \widetilde{M%
},\widetilde{g}\right).$ If $V$ is a torse-forming vector field on $%
\widetilde{M}$ and $\eta$ is the $g$-dual of $V^{T},$ then

i) $(M,g)$ is an almost $\eta $-Ricci soliton with potential vector field $%
V^{T}$ if and only if there exist two smooth functions $\lambda $ and $\mu $
on $M$ such that the Ricci tensor field of $M$ satisfies
\begin{equation}
\func{Ric}_M=\left( \lambda -a\right) g-\rho H+\mu \eta \otimes \eta-\frac{1%
}{2}(\psi\otimes \eta+\eta\otimes \psi);  \label{B2}
\end{equation}

ii) $(M,g)$ is an almost $\eta $-Yamabe soliton with potential vector field $%
V^{T}$ if and only if there exist two smooth functions $\lambda $ and $\mu $
on $M$ such that
\begin{equation}
(\func{scal}-\lambda -a)g-\rho H+\mu \eta \otimes \eta-\frac{1}{2}%
(\psi\otimes \eta+\eta\otimes \psi)=0.  \label{B23}
\end{equation}
\end{corollary}

\begin{theorem}
Let $M$ be an $n$-dimensional hypersurface isometrically immersed into an $%
(n+1)$-dimensional Riemannian manifold $\left( \widetilde{M}(c),\widetilde{g}%
\right) $ of constant curvature $c.$ If $V$ is a torse-forming vector field
on $\widetilde{M}$ and $\eta$ is the $g$-dual of $V^{T},$ then

i) $(M,g)$ is an almost $\eta $-Ricci soliton with potential vector field $%
V^{T}$ if and only if there exist two smooth functions $\lambda $ and $\mu $
on $M$ such that the second fundamental tensor field $H$ of $M$ satisfies
\begin{equation}
H^{2}=\left[ \rho +tr(H)\right] H+\left[ (n-1)c-\lambda +a) \right] g-\mu
\eta \otimes \eta +\frac{1}{2}(\psi \otimes \eta +\eta \otimes \psi);
\label{B3}
\end{equation}

ii) $(M,g)$ is an almost $\eta $-Yamabe soliton with potential vector field $%
V^{T}$ if and only if there exist two smooth functions $\lambda $ and $\mu $
on $M$ such that the second fundamental tensor field $H$ of $M$ satisfies
\begin{equation}
\rho H=[n(n-1)c+(tr(H))^2-tr(H^2)-\lambda-a]g+\mu\eta\otimes \eta-\frac{1}{2}%
(\psi \otimes\eta+\eta \otimes\psi).  \label{B55}
\end{equation}
\end{theorem}

\begin{proof}
From the Gauss equation, we have
\begin{equation}
\Ric_M(X,Y)=tr(H)H(X,Y)-H^{2}(X,Y)+(n-1)cg(X,Y).  \label{B4}
\end{equation}%

i) Then comparing (\ref{B4}) and (\ref{B2}), we get
\begin{equation*}
\left( \lambda -a\right) g(X,Y)-\rho H(X,Y)+\mu \eta (X)\eta
(Y)-\frac{1}{2}[\psi(X) \eta(Y)+\eta(X)\psi(Y)]$$$$=tr(H)H(X,Y)-H^{2}(X,Y)+(n-1)cg(X,Y),
\end{equation*}%
which gives us
\begin{equation*}
H^{2}(X,Y)=\left[ \rho +tr(H)\right] H(X,Y)+\left[ (n-1)c-\lambda
+a\right] g(X,Y)-\mu \eta (X)\eta (Y)$$$$+\frac{1}{2}[\psi(X) \eta(Y)+\eta(X)\psi(Y)].
\end{equation*}%

Conversely, assume that (\ref{B3}) is satisfied. Then by the Gauss equation, we
have
\begin{equation}
\Ric_M(X,Y)=\left( \lambda -a\right) g(X,Y)-\rho H(X,Y)+\mu \eta (X)\eta (Y)-\frac{1}{2}[\psi(X) \eta(Y)+\eta(X)\psi(Y)],
\label{B5}
\end{equation}%
so by Corollary \ref{corA}, $(M,g)$ is an almost $\eta $-Ricci soliton with
potential vector field $V^{T}.$

ii) By a contraction in (\ref{B4}), we find
$$\scal=(tr(H))^2-tr(H^2)+n(n-1)c$$
and replacing $\scal$ in (\ref{B23}), we get
$$[n(n-1)c+(tr(H))^2-tr(H^2)-\lambda-a]g-\rho H+\mu\eta\otimes \eta-\frac{1}{2}%
(\psi \otimes\eta+\eta \otimes\psi)=0.$$

Conversely, assume that (\ref{B55}) is satisfied. Then by a contraction, we
find
\begin{equation}
\scal=(tr(H))^2-tr(H^2)+n(n-1)c,
\end{equation}%
so by Corollary \ref{corA}, $(M,g)$ is an almost $\eta $-Yamabe soliton with
potential vector field $V^{T}.$
\end{proof}

\begin{proposition}
Let $M$ be an $n$-dimensional quasi-Einstein hypersurface isometrically
immersed into an $(n+1)$-dimensional Riemannian manifold $\left( \widetilde{M%
},\widetilde{g}\right) .$ Assume that the Ricci tensor field of $M$ is of
the form $\func{Ric}=\alpha g+\beta \eta \otimes \eta .$ If $V$ is a
concircular vector field on $\widetilde{M},$ then $(M,g)$ is an almost $\eta
$-Ricci soliton with potential vector field $V^{T}$, for $\eta $ the $g$%
-dual of $V^{T},$ if and only if it is a quasi-umbilical hypersurface with
associated functions $\frac{\lambda -a-\alpha }{\rho }$ and $\frac{\mu
-\beta }{\rho }.$
\end{proposition}

\begin{proof}
Assume that $M$ is a quasi-Einstein hypersurface whose Ricci tensor field $\Ric$ is
of the form $\Ric=\alpha g+\beta \eta \otimes \eta .$ If $V$ is a concircular
vector field on $\widetilde{M},$ then from Theorem \ref{thm-sub}, we can
write
\begin{equation*}
\alpha g(X,Y)+\beta \eta (X)\eta (Y)=\left( \lambda -a\right) g(X,Y)-\rho
H(X,Y)+\mu \eta (X)\eta (Y),
\end{equation*}%
which gives us
\begin{equation*}
H(X,Y)=\frac{\lambda -a-\alpha }{\rho }g(X,Y)+\frac{\mu -\beta }{\rho }\eta
(X)\eta (Y).
\end{equation*}%
Hence $M$ is a quasi-umbilical hypersurface with associated functions $\frac{%
\lambda -a-\alpha }{\rho }$ and $\frac{\mu -\beta }{\rho }.$
The converse is trivial.
\end{proof}

It is known that an $n$-dimensional hypersurface $M$, $n\geq 4,$ in a
Riemannian manifold $\left( \widetilde{M}(c),\widetilde{g}\right) $ of
constant curvature $c$ is conformally flat if and only if it is
quasi-umbilical \cite{Shou-21}. So we have:

\begin{corollary}
Let $M$ be an $n$-dimensional quasi-Einstein hypersurface isometrically
immersed into an $(n+1)$-dimensional Riemannian manifold $\left( \widetilde{M%
},\widetilde{g}\right) .$ Assume that the Ricci tensor field of $M$ is of
the form $\func{Ric}=\alpha g+\beta \eta \otimes \eta .$ If $V$ is a
concircular vector field on $\widetilde{M},$ then $(M,g)$ is an almost $\eta
$-Ricci soliton with potential vector field $V^{T}$, for $\eta $ the $g$%
-dual of $V^{T}$ and $M$ is a conformally flat hypersurface.
\end{corollary}

\bigskip

Let $\varphi :M\rightarrow \mathbb{S}^{n+1}(1)$ be an immersion. We denote
by $g$ the induced metric on the hypersurface $M$ as well as that on the
unit sphere $\mathbb{S}^{n+1}(1).$ Let $N$ and $B$ be the unit normal vector
field and the shape operator of the hypersurface $M$ in the unit sphere $%
\mathbb{S}^{n+1}(1)$ and we denote by $\left\langle ,\right\rangle $ the
Euclidean metric on the Euclidean space $\mathbb{E}^{n+2}.$ Assume that $V$
is a torse-forming vector field on $\mathbb{E}^{n+2}$. If we denote by $N_{%
\mathbb{S}}$ the unit normal vector field of the unit sphere $\mathbb{S}%
^{n+1}(1)$ in the Euclidean space $\mathbb{E}^{n+2}$, we can define the
smooth functions $\delta ,\varrho $ on the hypersurface $M$ by $\delta =$ $%
\left\langle V,N\right\rangle \mid _{M}$ and $\varrho $ $=\left\langle V,N_{%
\mathbb{S}}\right\rangle \mid _{M}.$ Hence the restriction of the
torse-forming vector field $V$ to the hypersurface $M$ can be written as $%
V\mid _{M}=U+\delta N+$ $\varrho N_{\mathbb{S}}$, where $U\in \chi (M).$

Then as an extension of Theorem 3.3 given in \cite{h}, we can state:

\begin{theorem}
\label{ThmB}Let $M$ be an orientable hypersurface of the unit sphere $%
\mathbb{S}^{n+1}(1),$ with immersion $\varphi :M\rightarrow \mathbb{S}%
^{n+1}(1)$ and let $V$ be a torse-forming vector field on the Euclidean
space $\mathbb{E}^{n+2}.$ Denote by $\xi $ the tangential component of $V$
to the unit sphere $\mathbb{S}^{n+1}(1)$ and by $U$ the tangential component
of $\xi $ on $M.$ Then

i) the hypersurface $(M,g)$ is an almost $\eta $-Ricci soliton $(U,\lambda
,\mu )$ if and only if
\begin{equation*}
H^{2}(X,Y)=\left[ tr(H)+\delta \right] H(X,Y)+(n-1-\varrho -\lambda +a)g(X,Y)
\end{equation*}%
\begin{equation*}
-\mu \eta (X)\eta (Y)+\frac{1}{2}\left[ \psi (X)g(U,Y)+\psi (Y)g(U,X)\right]
,
\end{equation*}%
for any $X,Y\in \chi (M)$;

ii) the hypersurface $(M,g)$ is an almost $\eta $-Yamabe soliton $(U,\lambda
,\mu )$ if and only if \
\begin{equation*}
\delta H(X,Y)=[n(n-1)+(tr(H))^{2}-tr(H^{2})+\varrho -\lambda -a]g(X,Y)
\end{equation*}%
\begin{equation*}
+\mu \eta (X)\eta (Y)-\frac{1}{2}\left[ \psi (X)g(U,Y)+\psi (Y)g(U,X)\right]
,
\end{equation*}%
for any $X,Y\in \chi (M)$. In this case, $M$ is a pseudo quasi-umbilical
hypersurface.
\end{theorem}

\begin{proof}
Let $\nabla ,\overline{\nabla }$ and $D$ denote the Levi-Civita connections
on $M$, $\mathbb{S}^{n+1}(1)$ and $\mathbb{E}^{n+2},$ respectively. Then we
can write
\begin{equation*}
V\mid _{\mathbb{S}^{n+1}(1)}=\xi +\varrho N_{\mathbb{S}},
\end{equation*}%
and for any $X\in \chi (M)$, by taking the covariant differential w.r.t. $X$%
, we have
\begin{equation*}
aX+\psi (X)V=aX+\psi (X)\xi +\psi (X)\varrho N_{\mathbb{S}}=D_{X}V
=D_{X}\xi +X(\varrho )N_{\mathbb{S}}+\varrho D_{X}N_{\mathbb{S}}.
\end{equation*}%
By using the Gauss and Weingarten formulas, we find
\begin{equation*}
aX+\psi (X)\xi +\psi (X)\varrho N_{\mathbb{S}}=\overline{\nabla }_{X}\xi
-g(X,\xi )N_{\mathbb{S}}+X(\varrho )N_{\mathbb{S}}+\varrho X.
\end{equation*}%
By the equality of the tangent and the normal components, we get%
\begin{equation}
\overline{\nabla }_{X}\xi +(\varrho -a)X=\psi (X)\xi   \label{recur1}
\end{equation}%
and
\begin{equation*}
X(\varrho )-g(X,\xi )=\psi (X)\varrho .
\end{equation*}%
The vector field $\xi $ on $\mathbb{S}^{n+1}(1)$ can be written as
\begin{equation*}
\xi =U+\delta N.
\end{equation*}%
So from (\ref{recur1}), we have
\begin{equation*}
\overline{\nabla }_{X}\left( U+\delta N\right) +(\varrho -a)X=\psi (X)\left(
U+\delta N\right) .
\end{equation*}%
By using Gauss and Weingarten formulas again, we find%
\begin{equation*}
\psi (X)U+\psi (X)\delta N-(\varrho -a)X=\nabla _{X}U+g(B(X),U)N+X(\delta
)N-\delta B(X).
\end{equation*}%
Then by the equality of the tangent and the normal components, we have%
\begin{equation}
\nabla _{X}U=\psi (X)U-(\varrho -a)X+\delta B(X)  \label{nablaU}
\end{equation}%
and
\begin{equation*}
\psi (X)\delta =g(B(X),U)+X(\delta ).
\end{equation*}%
So
\begin{equation*}
\left( \pounds _{U}g\right) (X,Y)=g(\nabla _{X}U,Y)+g(\nabla _{Y}U,X)
\end{equation*}%
\begin{equation}
=\psi (X)g(U,Y)+\psi (Y)g(U,X)-2(\varrho -a)g(X,Y)+2\delta H(X,Y).
\label{LieU}
\end{equation}%

On the other hand, the Gauss equation for a hypersurface $M$ in $\mathbb{S}%
^{n+1}(1)$ gives us
\begin{equation}
\func{Ric}(X,Y)=(n-1)g(X,Y)+tr(H)H(X,Y)-H^{2}(X,Y).  \label{Ric}
\end{equation}%

i) Then combining (\ref{Ric}) and (\ref{LieU}), we find
\begin{equation*}
\frac{1}{2}\left( \pounds _{U}g\right) \left( X,Y\right) +\func{Ric}%
(X,Y)=\left( n-1-\varrho +a\right) g(X,Y)+\left[ tr(H)+\delta \right] H(X,Y)
\end{equation*}%
\begin{equation*}
-H^{2}(X,Y)+\frac{1}{2}\left[ \psi (X)g(U,Y)+\psi (Y)g(U,X)\right] .
\end{equation*}%

\pagebreak
Suppose that there exist smooth functions $\lambda $ and $\mu $ on $M$ such
that the condition in the hypothesis holds. Then we obtain $\frac{1}{2}%
\left( \pounds _{U}g\right) \left( X,Y\right) +\func{Ric}(X,Y)=\lambda
g\left( X,Y\right) +\mu \eta (X)\eta (Y).$ Hence the hypersurface $M$ is an
almost $\eta $-Ricci soliton. The converse is trivial.

ii) By a contraction in (\ref{Ric}), we find
\begin{equation}
\scal=n(n-1)+\left( tr(H\right) )^{2}-tr( H^{2}) .  \label{scal}
\end{equation}%
Then combining (\ref{scal}) and (\ref{LieU}), we get
\begin{equation*}
\frac{1}{2}\left( \pounds _{U}g\right) \left( X,Y\right) -(\func{scal}%
-\lambda )g\left( X,Y\right) =\frac{1}{2}\left[ \psi (X)g(U,Y)+\psi (Y)g(U,X)%
\right]
\end{equation*}%
\begin{equation*}
-\left[ n(n-1)+\left( tr(H\right) )^{2}-tr( H^{2}) +\varrho -\lambda
-a\right] g(X,Y)+\delta H(X,Y).
\end{equation*}%

Suppose that there exist smooth functions $\lambda $ and $\mu $ on $M$ such
that the condition in the hypothesis holds. Then we obtain $\frac{1}{2}%
\left( \pounds _{U}g\right) (X,Y)=(\func{scal}-\lambda )g(X,Y)+\mu \eta
(X)\eta (Y)$. Hence the hypersurface $M$ is an almost $\eta $-Yamabe soliton
and moreover, $M$ is a pseudo quasi-umbilical hypersurface. The converse is
trivial.
\end{proof}

\begin{corollary}
Under the conditions of Theorem \ref{ThmB}, if the $1$-form $\psi $ is the $%
g $-dual of $U$, then

i) the hypersurface $(M,g)$ is an almost $\eta $-Ricci soliton $(U,\lambda
,\mu )$ if and only if
\begin{equation*}
H^{2}=\left[ tr(H)+\delta \right] H+(n-1-\varrho -\lambda +a)g-\mu \eta
\otimes \eta +\psi \otimes \psi ;
\end{equation*}

ii) the hypersurface $(M,g)$ is an almost $\eta $-Yamabe soliton $(U,\lambda
,\mu )$ if and only if
\begin{equation*}
\delta H=[n(n-1)+\left( tr(H\right) )^{2}-tr(H^{2})+\varrho -\lambda
-a]g+\mu \eta \otimes \eta -\psi \otimes \psi .
\end{equation*}
\end{corollary}

\begin{corollary}
Let $M$ be an orientable hypersurface of the unit sphere $\mathbb{S}%
^{n+1}(1) $ and let $V$ be a $\nabla $-parallel or constant vector field on
the Euclidean space $\mathbb{E}^{n+2}.$ Then

i) the hypersurface $(M,g)$ is an almost $\eta $-Ricci soliton $(U,\lambda
,\mu )$ if and only if
\begin{equation*}
H^{2}=\left[ tr(H)+\delta \right] H+(n-1-\varrho -\lambda )g-\mu \eta
\otimes \eta ;
\end{equation*}

ii) the hypersurface $(M,g)$ is an almost $\eta $-Yamabe soliton $(U,\lambda
,\mu )$ if and only if
\begin{equation*}
\delta H=[n(n-1)+\left( tr(H\right) )^{2}-tr(H^{2})+\varrho -\lambda ]g+\mu
\eta \otimes \eta .
\end{equation*}
\end{corollary}

It is known that a pseudo quasi-umbilical hypersurface of a Riemannian
manifold of constant curvature $\left( \widetilde{M}(c),\widetilde{g}\right)
$ is a pseudo quasi-Einstein hypersurface \cite{DeMa-2018}. So we have:

\begin{corollary}
\label{ThmB copy(3)}Let $M$ be an orientable hypersurface of the unit sphere
$\mathbb{S}^{n+1}(1)$ and let $V$ be a torse-forming vector field on the
Euclidean space $\mathbb{E}^{n+2}.$ If the hypersurface $(M,g)$ is an almost
$\eta $-Yamabe soliton $(U,\lambda ,\mu )$, then it is a pseudo
quasi-Einstein hypersurface.
\end{corollary}

%\textbf{Acknowledgements.} The authors are grateful to the referees for the valuable suggestions and remarks that definitely improved the paper.

\bigskip

Adara M. BLAGA

Department of Mathematics, West University of Timi\c{s}oara

300223, Bld. V. P\^{a}rvan nr. 4, Timi\c{s}oara, Rom\^{a}nia

Email: adarablaga@yahoo.com

\bigskip

Cihan \"{O}ZG\"{U}R

Department of Mathematics, Bal\i kesir University

10145, \c{C}a\u{g}\i \c{s}, Bal\i kesir, Turkey

Email: cozgur@balikesir.edu.tr

\end{document}